\documentclass[a4paper, fleqn, 12pt]{article}
\usepackage[a4paper, left=3cm, right=3cm]{geometry}
\usepackage[utf8]{inputenc}
\usepackage[english]{babel}
\usepackage[T1]{fontenc}
\usepackage{mathtools}
\usepackage{amssymb}
\usepackage{paralist}
\usepackage{tikz}
\usepackage{amsmath}
\usepackage{mathrsfs}
\usepackage{tikz-cd}
\usepackage{amsthm}
\usepackage{lscape}
\usepackage{enumitem}
\usetikzlibrary{arrows}

\newtheorem{theorem}{Theorem}[section]
\newtheorem{lemma}[theorem]{Lemma}
\newtheorem{corollary}[theorem]{Corollary}

\newtheorem{definition}[theorem]{Definition}
\newtheorem{remark}[theorem]{Remark}
\newtheorem{proposition}[theorem]{Proposition}

\numberwithin{subcase}{case}
\DeclareMathOperator{\Irr}{Irr}
\DeclareMathOperator{\PSL}{PSL}
\DeclareMathOperator{\SU}{SU}
\DeclareMathOperator{\SL}{SL}
\DeclareMathOperator{\PSU}{PSU}

\begin{document}
\title{Orthogonal Determinants of $\SL_3$(q) and $\SU_3$(q)}
\author{Linda Hoyer\footnote{linda.hoyer@rwth-aachen.de} \ and Gabriele Nebe\footnote{nebe@math.rwth-aachen.de}}
\date{Lehrstuhl f\"ur Algebra und Zahlentheorie, RWTH Aachen University, Germany}
\maketitle
\begin{abstract}
We give a full list of the orthogonal determinants of the even degree indicator '+' ordinary irreducible characters of $\mathrm{SL}_3(q)$ and $\mathrm{SU}_3(q)$.
\\
 {\sc Keywords}:  Orthogonal representations, invariant quadratic forms, generic orthogonal character table, finite groups of Lie type. 
{\sc MSC}: 20C15, 11E12.
\end{abstract}

\section{Introduction}
Let $G$ be a finite group and $\rho:G \to \mathrm{GL}_n(\mathbb{C})$ be a representation. We call $\rho$ an orthogonal representation, if there is a symmetric, non-degenerate, $\rho(G)$-invariant bilinear form $\beta$ on ${\mathbb C}^n$.
It is well-known that being an irreducible orthogonal representation is equivalent to the associated character having Frobenius-Schur indicator '+', i.e. $\rho$ being equivalent to a  real representation. A character is called orthogonal if it is the character afforded by an orthogonal representation. The character $\chi $ is orthogonal, if and only if it is of the form 
\begin{equation}\label{characterform}
\chi=\sum_{i=1}^r a_i \chi^{(+)}_i + 2\sum_{j=1}^s b_i \chi^{(-)}_j + \sum_{k=1}^t c_i (\chi^{(0)}_k+ \overline{\chi_k}^{(0)}),
\end{equation}
where $\chi_i^{(+)}$ (resp. $\chi^{(-)}_j$, resp. $\chi^{(0)}_k$) are irreducible characters of $G$ with Frobenius-Schur indicator '+' (resp. '-', resp. '0'), and $a_i, b_j, c_k$ are non-negative integers.

Let $\chi$ be an orthogonal character as in equation \eqref{characterform} 
and let $K=\mathbb{Q}(\chi)$ be the character field of $\chi$. The main result of  \cite{NebeOrtDet} shows that if the degree of all $\chi_i^{+}$ is even then there is a unique element $\det(\chi) \coloneqq d \in K^{\times}/(K^{\times})^2$, called the \textit{orthogonal determinant} of $\chi$, such that for all representations $\rho:G \to \mathrm{GL}_n(L)$ over a field extension $L \supseteq K$ affording $\chi$  all non-degenerate, $\rho(G)$-invariant, symmetric  bilinear forms $\beta $ on $L^n$ have the same determinant 
$$\det(\beta )=d \cdot (L^{\times})^2 \in L^{\times}/(L^{\times})^2.$$  We call such a character \textit{orthogonally stable}.

The orthogonal determinant of  $2\chi_j^{(-)}$ is always $1$ (see \cite{zbMATH00238145}) 
and for characters of the form 
$\chi_k^{(0)}+\overline{\chi_k}^{(0)}$ the orthogonal determinant can be obtained from the character values as given in Lemma \ref{complex}.
So it remains to  deal with the indicator '+' part.
Put $$ \mathrm{Irr}^+(G) = \{\chi \in \mathrm{Irr}(G) \mid  \chi \ \text{is an orthogonal character of even degree} \}.$$
In a long term project theoretical and computational methods are used to calculate the 
orthogonal determinants of the small finite simple groups (see \cite{orthatlas} for a survey).
The goal of this paper is to determine the orthogonal determinants 
of the characters in $\mathrm{Irr}^+(G)$ for the two infinite series of 
finite groups of Lie type, $G=\mathrm{SL}_3(q)$ and $G=\mathrm{SU}_3(q)$, for all prime powers $q$. 
For $\mathrm{SL}_2(q)$ the orthogonal character table is already computed in \cite{SL2}.  
An important subgroup to analyse the structure and the representations of a finite group $G$ of Lie type is its standard Borel subgroup $B$. A character $\chi \in \Irr ^+(G) $ is called
Borel stable, if the restriction of $\chi $ to $B$ is orthogonally stable. 
The structure of $B$ as a semidirect product of a $p$-group and 
a split torus $T$ allows us to determine the orthogonal determinants of all Borel stable characters of $G$
(see Remark \ref{RemChiTChiU}). For the groups 
$G=\mathrm{SL}_3(q)$ and $G=\mathrm{SU}_3(q)$ it turns out that 
those $\chi \in \Irr^+(G)$ that are not Borel stable 
occur nicely in a permutation character. 
For such characters Lemma \ref{Permutation repr} 
can be used to determine their orthogonal determinants. 

This paper is a contribution to 
Project-ID 286237555 – TRR 195 – by the Deutsche Forschungsgemeinschaft
(DFG, German Research Foundation).

\section{Methods}
This section collects some basic results about orthogonal determinants.

\begin{definition}
Let $G$ be a finite group, $H \subseteq G$ a subgroup, and let $\chi$ be an orthogonal character of $G$.  Then $\chi$ is called \textit{$H$-stable} 
if the restriction $\mathrm{Res}^G_H(\chi)$  of $\chi $ to $H$ is an orthogonally stable character of $H$. 
\end{definition}

\begin{lemma} (see \cite[Proposition 5.17 and Remark 5.21]{Nebe2022OrthogonalS} for a more general statement of (ii))
\label{Sum-Lemma}
\begin{enumerate}[label=(\roman*)]
    \item If $\chi$ is $H$-stable, then $\chi $ is orthogonally stable and $$\det(\chi)=\det(\mathrm{Res}^G_H(\chi)) \cdot (\mathbb{Q}(\chi)^{\times})^2.$$
    \item If $\chi=\sum_{i=1}^k \chi_i$ is the sum of orthogonally stable characters $\chi _i$ 
    then $\chi $ is orthogonally stable. 
Moreover if $\mathbb{Q}(\chi_i) \subseteq \mathbb{Q}(\chi)$ for all $i$ then $$\det(\chi) = \prod_{i=1}^k \det(\chi_i) \cdot (\mathbb{Q}(\chi)^{\times})^2.$$
\end{enumerate}
\end{lemma}

The paper \cite{10.51286/albjm/1658730113} gives an easy formula for 
the orthogonal determinant of orthogonally stable characters of $p$-groups.
We only need the following special case:

\begin{lemma} (see \cite[Corollary 4.4]{10.51286/albjm/1658730113})
\label{p group Prop}
Let $p$ be an odd prime and let $\chi $ be an orthogonally stable rational character of a finite 
$p$-group. Then $p-1$ divides $\chi (1)$ and 
     $\det(\chi)=p^{\chi(1)/(p-1)} \cdot (\mathbb{Q}^{\times})^2$.
\end{lemma}

\begin{lemma}  (see \cite[Proposition 3.12]{Nebe2022OrthogonalS})
\label{complex}
Let $\psi=\chi+\overline{\chi}$ for some indicator '0' irreducible character $\chi$. Let $K=\mathbb{Q}(\psi)$, $L=\mathbb{Q}(\chi)$. Then $K$ is the maximal real subfield of the complex field $L$. So there is a totally positive  $\delta \in K$ such that $L=K[\sqrt{-\delta}]$. Then
$$
\det(\psi)=\delta^{\chi(1)} \cdot (K^{\times})^2.
$$
For the cyclotomic fields
$L=\mathbb{Q}(\mu_{m}^j)$, 
where $ 
\mu_m \coloneqq \exp \left( \frac{2\pi i}{m} \right) \in \mathbb{C} 
$, we get 
$$K=\mathbb{Q}(\vartheta_{m}^{(j)}) \mbox{ with }
\vartheta_m^{(j)} \coloneqq \mu_m^j+\mu_m^{-j} \in \mathbb{R} $$  
and can choose $$\delta=2-\vartheta^{(2j)}_m=-(\mu^j_m-\mu^{-j}_m)^2.$$ 
\end{lemma}

\begin{lemma}
\label{Permutation repr}
Let $G$ be a finite group acting on a finite set $M$. Let $V$ be the permutation representation over $\mathbb{Q}$. Define the $G$-invariant 
bilinear form $\beta :V \times V \to \mathbb{Q}$ by choosing $M$ to be an orthonormal basis. 
Then $V_1=\langle \sum_{m \in M} m \rangle$  and $V_1^{\bot}$ are $G$-invariant subspaces 
and  $$ \det(\beta _{| V_1^{\bot}})=|M| \cdot (\mathbb{Q}^{\times})^2.$$
\end{lemma}
\begin{proof}
It is clear that $\det(\beta )=\det(\beta _{| V_1}) \cdot \det(\beta _{| V_1^{\bot}}) = 1 \cdot (\mathbb{Q}^{\times})^2$ and $$\beta  \left( \sum_{m \in M} m, \sum_{m \in M} m \right) =|M|,$$ from which the result follows.
\end{proof}

\section{The Orthogonal Characters of $\SL_3(q)$ and $\SU_3(q)$}

Let $p$ be a prime and let $q$ be a power of $p$. The group $\mathrm{SL}_3(q)$ is
$$
\mathrm{SL}_3(q) = \{ A \in \mathbb{F}_q^{3 \times 3} | \det(A)=1 \}.
$$
Let $\mathbb{F}_{q^2} \supseteq \mathbb{F}_q$ be the field extension of degree $2$ and let $F:\mathbb{F}_{q^2} \to \mathbb{F}_{q^2}$, $x \mapsto x^q$ be the Frobenius automorphism. For a matrix $A \in \mathbb{F}_{q^2}^{n \times m}$ we define $F(A)$ to be the matrix where we apply $F$ component-wise. Let
$$\Omega= \begin{pmatrix}
0 & 0 & 1 \\
0 & 1 & 0 \\
1 & 0 & 0
\end{pmatrix}$$
and define the Hermitian form $H: \mathbb{F}_{q^2}^{3} \times \mathbb{F}_{q^2}^{3} \to \mathbb{F}_{q^2}$, $H(v,w) =F(v)^{tr} \cdot \Omega \cdot w$. Then
\begin{align*}
    \mathrm{SU}_3(q) = &\{ A \in \mathrm{SL}_3(q^2) | H(A \cdot v, A \cdot w)=H(v,w) \ \text{for all} \ v,w \in \mathbb{F}_{q^2}\}= \\ &\{ A \in \mathrm{SL}_3(q^2)| F(A)^{tr} \cdot \Omega \cdot A = \Omega \}.
\end{align*}

The letter $G$ will always denote one of $\SL_3(q)$ or $\SU_3(q)$. 
The full character table of $G$ was first calculated in \cite{Simpson1973TheCT} in 1973 by Simpson and Frame. By  'Ennola duality' (see  \cite{Ennola} for the statement and \cite{EnnolaProof} for a proof),
$$"\mathrm{SU}_3(q)=\mathrm{SL}_3(-q)",$$
the irreducible characters  of $\mathrm{SU}_3(q)$ can be obtained from the ones of $\mathrm{SL}_3(q)$ by formally replacing every instance of  $q$ by $-q$, so that there is a single generic character table giving the  character table for both groups introducing an additional parameter 
$\varepsilon=+1$ for $G=\mathrm{SL}_3(q)$ and $\varepsilon=-1$ for $G=\mathrm{SU}_3(q)$.

In this notation the center of $G$ is the group of scalar matrices in $G$ and hence 
of order $\gcd(q-\varepsilon , 3)$. In particular the set $\Irr^+(G)$  is the 
set of irreducible orthogonal characters of even degree of the group
\begin{align*}
 \mathrm{PSL}_3(q)=\SL_3(q)/Z(\SL_3(q))  \mbox{ and } 
\mathrm{PSU}_3(q)=\SU_3(q)/Z(\SU_3(q)).
\end{align*}
 It is well known that, with the exception of $\mathrm{PSU}_3(2)$,  the groups $\mathrm{PSL}_3(q)$ and $\mathrm{PSU}_3(q)$ are simple groups for all $q$.
The irreducible characters of $\PSU_3(q)$ and $\PSL_3(q)$ are the irreducible  characters of $G$ that are constant on the center.

Gow \cite{GOW1976102} showed that all the characters of $\mathrm{PSL}_3(q)$ and $\mathrm{PSU}_3(q)$ have Schur index 1, except the unique character of degree $q^2-q$ of $\mathrm{PSU}_3(q)$, which has Schur index 2 and Frobenius-Schur indicator '-'. Additionally, the results in \cite{FieldValues} allow us to obtain the character fields from some combinatorial description.   
For cyclotomic numbers we use the notation from Lemma \ref{complex}
and for the naming convention of the irreducible characters we follow \cite[Table 2]{Simpson1973TheCT}.
Then the set $\mathrm{Irr}^+(G)$ is given as follows:
\begin{theorem}
\label{Irr+ list}
The following table includes all $\chi \in \mathrm{Irr}^+(G)$, their degrees $\chi (1)$  and character fields $\mathbb{Q}(\chi )$:
\begin{center}
\begin{tabular}{|c || c |c |c|} 
 \hline
 $\chi$ & $u$ & $\chi(1)$ & $\mathbb{Q}(\chi)$ \\  
 \hline\hline
 $\chi_{qs}$ & --- & $q(q+\varepsilon)$ & $\mathbb{Q}$ \\
 \hline
$\chi_{q^3}$ & --- & $q^3$ & $\mathbb{Q}$ \\ 
 \hline
$\chi_{st'}^{(u)}$ & $0 \leq u \leq 2 $ & $1/3 (q+\varepsilon)(q^2+\varepsilon q +1)$ & $\mathbb{Q}$ \\
 \hline
 $\chi_{st}^{(u,-u,0)}$ & \begin{tabular}{@{}c@{}} $1 \leq u < q- \varepsilon$, \\ $u \notin \{ (q-\varepsilon)/3, (q-\varepsilon)/2, 2(q-\varepsilon)/3 \}$ \end{tabular} & $(q+\varepsilon)(q^2+\varepsilon q +1)$ & $\mathbb{Q}(\vartheta_{q-\varepsilon}^{(u)})$ \\
 \hline
$\chi_{rt}^{((q-\varepsilon)u)}$ & $ 1 \leq u < q+\varepsilon$ & $(q-\varepsilon)(q^2+\varepsilon q +1)$ & $\mathbb{Q}(\vartheta_{q+\varepsilon}^{(u)})$ \\
 \hline
\end{tabular}
\end{center}
\begin{itemize}
    \item For $q$ odd and $G=\SL_3(q)$, $\mathrm{Irr}^+(G)=\{ \chi_{qs}, \chi_{st'}^{(u)}, \chi_{st}^{(u,-u,0)}, \chi_{rt}^{((q-\varepsilon)u)} \}$.
    \item For $q$ odd and $G=\SU_3(q)$, $\mathrm{Irr}^+(G)=\{\chi_{st'}^{(u)}, \chi_{st}^{(u,-u,0)}, \chi_{rt}^{((q-\varepsilon)u)} \}$.
    \item For $q$ even and $G=\SL_3(q)$, $\mathrm{Irr}^+(G)=\{ \chi_{qs}, \chi_{q^3} \}$.
    \item For $q$ even and $G=\SU_3(q)$, $\mathrm{Irr}^+(G)=\{  \chi_{q^3} \}$.
\end{itemize}
Note that the characters $\chi_{st'}^{(u)}$ only exist for $3 | q- \varepsilon$.
\end{theorem}

\section{Results}

Let $G=\mathrm{SL}_3(q)$ or $G=\mathrm{SU}_3(q)$. Let
$$
B \coloneqq \left\{   
\begin{pmatrix}
d & a & b \\
0 & e & c \\
0 & 0 & f
\end{pmatrix} \in G
\right\} \mbox{ and }
U \coloneqq \left\{   
\begin{pmatrix}
1 & a & b \\
0 & 1 & c \\
0 & 0 & 1
\end{pmatrix} \in G
\right\}.
$$
Then $U$ is the unipotent radical of $B$ and a Sylow $p$-subgroup of $G$, and $B=N_G(U) = U \rtimes T$ 
is a (standard) Borel subgroup, 
where $T:=\{\mathrm{diag}(d,e,f)\in G\}$ is a maximal torus. 
Denote by $W=N_G(T)/T$ the Weyl group of $G$.

We need an explicit notation for $\Irr(T)$: 
\\
For $G=\SL_3(q)$ we fix a generator $t$ of $\mathbb{F}_q^{\times }$.
Then the torus $$T = \{ t_{a,b}:=\mathrm{diag} (t^a,t^{-a-b},t^b) \mid  a,b \in \{ 0,\ldots , q-2 \} \}$$ 
is isomorphic to $\mathbb{F}_q^{\times } \times \mathbb{F}_q^{\times } $ and 
\begin{equation} \label{SLTorusIrr} 
\Irr(T) =\{ \alpha _1^{u_1} \alpha_2 ^{u_2} : t_{a,b} \mapsto \mu_{q-1}^{au_1+bu_2} \mid  u_1,u_2 \in \{ 0,\ldots , q-2 \} \} .
\end{equation}
For $G=\SU_3(q)$ the torus is isomorphic to $\mathbb{F}_{q^2}^{\times } =: \langle \tau \rangle $. So
$$ T = \{ \tau _a := \mathrm{diag} (\tau^a, \tau^{(q-1)a} ,\tau^{-qa }) \mid a\in \{0,\ldots , q^2-2 \} \}$$ and 
\begin{equation} \label{SUTorusIrr} 
	\Irr(T) =\{ \alpha ^{u} : \tau_a \mapsto \mu_{q^2-1}^{au} \mid u \in \{0,\ldots , q^2-2 \} \} .
\end{equation}
To unify notation we put 
\begin{equation} \label{theta}
\theta ^{(u) } := \left\{ \begin{array}{ll}   \alpha_1^u\alpha_2^{-u} & G = \SL_3(q) \\ \alpha^u & G = \SU_3(q) . \end{array} \right.  
\end{equation}

\begin{remark} 
By Lemma \ref{Sum-Lemma}  (i) the orthogonal determinant of 
a $B$-stable character $\chi$ of $G$ is determined by the restriction
 of $\chi $ to $B$. 
 Decompose this restriction as 
 \begin{equation} \label{chiTchiU} 
 \mathrm{Res}^G_B(\chi)=\chi_T+\chi_U \end{equation} 
 where $\chi_T$ is the character of $T$ on the $U$-fixed space, also known as the \textit{Harish-Chandra restriction} of $\chi $. In particular its degree is $\chi_T(1)=\langle \mathrm{Res}^G_U(\chi), \mathbf{1}_U \rangle. $
\end{remark}
Note that for odd primes $p$ the character $\chi _U$ of $B$ is $U$-stable. As 
$p$ does not divide the discriminant of $\mathbb{Q}(\chi ) $ for all 
$\chi \in \mathrm{Irr}^+(G) $ we have $\mathbb{Q}({\chi_U})=\mathbb{Q}$
so Lemma \ref{p group Prop} gives the 
determinant of $\chi _U$:

\begin{remark} \label{RemChiTChiU}
Let $q$ be odd. If $\chi _T$ is orthogonally stable then 
$$\det(\chi ) = \det(\chi _T) p^{\chi_U(1)/(p-1)} $$
\end{remark}

Note that $T$ is abelian and so $\chi _T$ is a sum of linear characters.
If these characters are complex (i.e. of indicator '0') then $\chi_T$ is orthogonally stable and 
its determinant  can be computed from Lemma \ref{complex}.  
In fact the irreducible constituents of $\chi _T$ can be obtained from the action of the Weyl group $W$ 
on $\mathrm{Irr}(T)$.
It is well-known that 
$$ W \cong \left\{ \begin{array}{ll}    S_3 & \mbox{ for }  G=\mathrm{SL}_3(q)  \\ 
C_2  & \mbox{ for } G=\mathrm{SU}_3(q) \end{array} \right. $$

Let $\theta \in \mathrm{Irr}(T)$. Then $\theta$ can also be considered as a character of $B$. A character $\chi \in \mathrm{Irr}(G)$ is said to be in the principal series if $\chi$ appears in $\mathrm{Ind}^G_B(\theta)$ for some $\theta \in \mathrm{Irr} (T)$.  

We will need a special case of the well-known Mackey formula for Harish-Chandra induction and restriction:
\begin{lemma} (see \cite[Theorem 5.2.1]{digne_michel_2020})
\label{Mackey Formula}
$$
(\mathrm{Ind}^G_B(\theta))_T=\sum_{w \in W} w \cdot \theta
.$$
\end{lemma}

\begin{corollary}
\label{Degree Characters}
Let $\chi \in \Irr(G)$. Then $0 \leq \chi_T(1) \leq |W|$, with $\chi_T(1)=0$ if and only if $\chi$ is not in the principal series, and $\chi_T(1)=|W|$ if and only if $\chi=\mathrm{Ind}^G_B(\theta)$ for some $\theta \in \Irr(T)$.
\end{corollary}

Explicit calculations with the character table \cite[Table 2]{Simpson1973TheCT} now give rise to the following propositions:
\begin{proposition}
\label{Main PropositionSL}
Let $G=\mathrm{SL}_3(q)$. The only characters in $\mathrm{Irr}^+(G)$ which are not in the principal series are $\chi_{rt}^{((q-1)u)}$ for $q$ odd. 
    \begin{itemize}
        \item[(a)]   $\theta^{(0)}= \mathbf{1}_T$ is the trivial character and  $$ \mathrm{Ind}^G_B(\theta^{(0)})= \mathbf{1}_G+2\chi_{qs} +\chi_{q^3}.$$
        \item[(b)] For $1 \leq u < q- 1$, $u \notin \{ (q-1)/3,(q-1)/2,  2(q-1)/3 \}$, we have that $$ \mathrm{Ind}^G_B(\theta^{(u)})=\chi_{st}^{(u,-u,0)}$$ and 
        $(\chi_{st}^{(u,-u,0)})_T(1)=6$, i.e. the $U$-fixed space 
        in $\chi_{st}^{(u,-u,0)}$ has dimension 6.
        \item[(c)] For $j \in \{ (q-1)/3, 2(q-1)/3 \}$, we have that $$ \mathrm{Ind}^G_B(\theta^{(j)})= \sum_{u=0}^2\chi_{st'}^{(u)}$$ 
        where $(\chi_{st'}^{(u)})_T(1)=2$ for $u=0,1,2$.
    \end{itemize}
    \end{proposition}
\begin{proposition}
\label{Main PropositionSU}
 Let $G=\mathrm{SU}_3(q)$. The only characters in $\mathrm{Irr}^+(G)$ which are not in the principal series are $\chi_{st}^{(u,-u,0)}$ and $\chi_{st'}^{(u)}$ for $q$ odd. 
    \begin{itemize}
        \item[(a)]  $\theta^{(0)}= \mathbf{1}_T$ is the trivial character and  $$ \mathrm{Ind}^G_B(\theta^{(0)})= \mathbf{1}_G+\chi_{q^3}.$$
        \item[(b)] For $1 \leq u < q+1$, we have that $$ \mathrm{Ind}^G_B(\theta^{((q+1)u)})=\chi_{rt}^{((q+1)u)}$$ and $(\chi_{rt}^{((q+1)u)})_T(1)=2$.
    \end{itemize}
\end{proposition}
\begin{proof}
We will handle both $G=\mathrm{SL}_3(q)$ and $G=\mathrm{SU}_3(q)$ simultaneously. Let $d=\gcd(q-\varepsilon,3)$. There are $2+d$ conjugacy classes of $G$ which have a non-empty intersection with $U$: $C_1$, $C_2$ and $C_3^{(l)}$, $0 \leq l \leq d-1$, which can be characterised by $\mathrm{rank}(g_i-I_3)=i-1$ for $g_i \in C_i, 1 \leq i \leq 3$. We further calculate that $|C_1 \cap U|=1$, and
\begin{align*}
&|C_2 \cap U|=  \left\{ \begin{array}{ll}   2q^2-q-1 & \mbox{ for }  G=\mathrm{SL}_3(q)  \\ 
q-1 & \mbox{ for } G=\mathrm{SU}_3(q) \end{array} \right. \\
&|C_3^{(l)} \cap U| = \left\{ \begin{array}{ll}   1/d(q^3-2q^2+q) & \mbox{ for }  G=\mathrm{SL}_3(q)  \\ 
1/d(q^3-q) & \mbox{ for } G=\mathrm{SU}_3(q) . \end{array} \right. 
\end{align*}

We will, as an example, calculate $(\chi_{st}^{(u,-u,0)})_T(1)=\langle \mathrm{Res}^G_U(\chi_{st}^{(u,-u,0)}), \mathbf{1}_U \rangle$. For $\SL_3(q)$, we see that
$$
(\chi_{st}^{(u,-u,0)})_T(1)=\frac{1}{q^3} \left(  (q+1)(q^2+q+1)+(2q^2-q-1)(2q+1) +(q^3-2q^2+q) \right)=6,
$$
whereas for $\SU_3(q)$, the calculation becomes
$$
(\chi_{st}^{(u,-u,0)})_T(1)=\frac{1}{q^3} \left(  (q-1)(q^2-q+1)+(q-1)(2q-1) -(q^3-q) \right)=0.
$$
The rest of the propositions is handled analogously.
\end{proof}

We are now ready to give the main result.

\begin{theorem}
\begin{itemize}
    \item[(i)] Let $q$ be odd. 
\begin{center}
\begin{tabular}{|c | c |c|} 
 \hline
 $\det(\chi)$ for $G=\mathrm{SL}_3(q)$ & $\chi$ & $\det(\chi)$ for $G=\mathrm{SU}_3(q)$ \\  
 \hline\hline
 $3q \cdot (\mathbb{Q}^{\times})^2$  & $\chi_{st'}^{(u)}$ & $q \cdot (\mathbb{Q}^{\times})^2$ \\
 \hline
 $q(2-\vartheta_{q-1}^{(2u)}) \cdot (\mathbb{Q}(\vartheta_{q-1}^{(u)})^{\times})^2$  & $\chi_{st}^{(u,-u,0)}$ & $q \cdot (\mathbb{Q}(\vartheta_{q+1}^{(u)})^{\times})^2$ \\
 \hline
 $q \cdot (\mathbb{Q}(\vartheta_{q+1}^{(u)})^{\times})^2$  & $\chi_{rt}^{((q-\varepsilon)u)}$ & $q (2-\vartheta_{q-1}^{(2u)})\cdot (\mathbb{Q}(\vartheta_{q-1}^{(u)})^{\times})^2$ \\
 \hline
\end{tabular}
\end{center}
Additionally, for $G=\mathrm{SL}_3(q)$, $\det(\chi_{(qs)})=(q^2+q+1)\cdot (\mathbb{Q}^{\times})^2$.
\item[(ii)] Let $q$ be even. 
\begin{center}
\begin{tabular}{|c | c |c|} 
 \hline
 $\det(\chi)$ for $G=\mathrm{SL}_3(q)$ & $\chi$ & $\det(\chi)$ for $G=\mathrm{SU}_3(q)$ \\  
 \hline\hline
 $(q+1)(q^2+q+1) \cdot (\mathbb{Q}^{\times})^2$  & $\chi_{q^3}$ & $(q^3+1) \cdot (\mathbb{Q}^{\times})^2$ \\
 \hline
\end{tabular}
\end{center}
Additionally, for $G=\mathrm{SL}_3(q)$, $\det(\chi_{(qs)})=(q^2+q+1)\cdot (\mathbb{Q}^{\times})^2$.
\end{itemize}
\end{theorem}
\begin{proof}
We first deal with the case that $q$ is odd: 
Then, for all $\chi \in \mathrm{Irr}^+(G) \backslash \{ \chi_{qs} \}$,
	the value $\chi_U(1)/(p-1)$ is even if and only if 
$q$ is an even power of $p$. So 
Lemma \ref{p group Prop} gives us that 
$\det(\chi_{U})=q \cdot (\mathbb{Q}^{\times})^2$.

For the characters 
$\chi_{st}^{(u,-u,0)}$ of $\mathrm{SL}_3(q)$ Lemma \ref{Mackey Formula} and Proposition \ref{Main PropositionSL} give that  $$(\chi_{st}^{(u,-u,0)})_T=\theta_1+\overline{\theta}_1 +  \theta_2+\overline{\theta}_2 + \theta_3+\overline{\theta}_3$$ is the sum of $6$ one-dimensional characters of $T$, where $\theta_1=\theta^{(u)}$,
$\theta_2=\alpha_{1}^{2u} \alpha_{2}^{u}$,  and $\theta_3=\alpha_{1}^{u} \alpha_{2}^{2u}$.
By Lemma \ref{Sum-Lemma} and Lemma \ref{complex} 
$$\det((\chi_{st}^{(u,-u,0)})_T)=(2-\vartheta_{q-1}^{(2u)}) \cdot (\mathbb{Q}(\vartheta_{q-1}^{(u)})^{\times})^2$$ 
and
\begin{align*}
\det(\chi_{st}^{(u,-u,0)})=&\det((\chi_{st}^{(u,-u,0)})_T) \det((\chi_{st}^{(u,-u,0)})_U) \cdot (\mathbb{Q}(\vartheta_{q-1}^{(u)})^{\times})^2= \\
& q(2-\vartheta_{q-1}^{(2u)}) \cdot (\mathbb{Q}(\vartheta_{q-1}^{(u)})^{\times})^2.    
\end{align*}
The results for $\chi_{rt}^{((q-\varepsilon)u)}$ for $\mathrm{SU}_3(q)$ and $\chi_{st'}^{(u)}$ follow similarly. Note here that $$2-\vartheta_{q-1}^{(2 (q-1)/3)}=2-(-1)=3.$$

The character $\chi_{qs}$ for $G=\mathrm{SL}_3(q)$  ($q$ even or odd) occurs in the permutation character $\psi := 1_P^G$ induced from the parabolic subgroup 
$$
P \coloneqq \left\{   
\begin{pmatrix}
a & b & c \\
d & e & f \\
0 & 0 & g
\end{pmatrix} \in G
\right\}
.$$  It is not hard to see that $|G/P|=q^2+q+1$ and that $\psi=\mathbf{1}_G+\chi_{qs}$ (cf \cite{Steinberg}).  
Lemma \ref{Permutation repr} hence yields $$\det(\chi_{qs})=(q^2+q+1) \cdot (\mathbb{Q}^{\times})^2.$$

To finish the proof let $q$ be even and regard the characters $\chi_{q^3}$. 
In both cases these appear in the permutation character $\phi := 1_B^G$. 

For $G=\mathrm{SU}_3(q)$ we get  $|G/B|=q^3+1$ and $\phi = \mathbf{1}_G+\chi_{q^3}$
so $\det(\chi_{q^3}) = q^3+1$ 
by Lemma \ref{Permutation repr}. 

For $G=\mathrm{SL_3}(q)$ we get $|G/B|=(q+1)(q^2+q+1)$ and $\phi = \mathbf{1}_G+2\chi_{qs}+\chi_{q^3}$. As $\chi _{qs}$ is orthogonally stable we get 
$$(q+1)(q^2+q+1) \cdot (\mathbb{Q}^{\times})^2 = \det(\chi_{qs})^2\det(\chi_{q^3}) \cdot (\mathbb{Q}^{\times})^2 = \det(\chi_{q^3}) \cdot (\mathbb{Q}^{\times})^2,$$ which finishes the proof.
\end{proof}

\bibliography{References}
\end{document}